\documentclass[11pt,reqno]{amsart}
\usepackage{comment,mathtools,amsmath,amssymb,mathtools,amsthm,overpic,graphicx}
\usepackage{comment}
\newtheorem{thm}{Theorem}

\newtheorem{cor}[thm]{Corollary}
\newtheorem{prop}[thm]{Proposition}
\theoremstyle{definition}
\newtheorem*{rem}{Remark}
\newtheorem*{question}{Question}

\newtheorem{definition}[thm]{Definition}
\usepackage[
	backend=biber
]{biblatex}
\title{$T$-Polynomial Convexity and Holomorphic Convexity}
\author{Blake J. Boudreaux}
\address[Blake J. Boudreaux]{The University of Western Ontario}
\email{bboudre7@uwo.ca}
\date{}
\begin{document}
\maketitle
\begin{abstract}
	We compare the $T$-polynomial convexity of Guedj with holomorphic convexity away from the support of $T$. In particular we show an Oka--Weil theorem for $T$-polynomial convexity, as well as present a situation when the notions of $T$-polynomial convexity and holomorphic convexity of $X\setminus\text{Supp }T$ coincide in the context of complex projective algebraic manifolds.
\end{abstract}

\section{Introduction and Preliminaries}
A compact set $K\subset\mathbb{C}^n$ is said to be polynomially convex if it agrees with its hull
\begin{equation}\label{PolyHull}
	\widehat{K}_{\mathbb{C}^n}=\left\{z\in X\,:\,|P(z)|\leq\sup_K|P|\text{ for all holomorphic polynomials $P$}\right\}.
\end{equation}
Polynomial convexity can be generalized to Stein manifolds in a straightforward manner by replacing polynomials with entire functions in the above definition and retains many of the same properties. This is appropriately dubbed holomorphic convexity~\cite{FoFoWo20}. On the other hand, there is no intrinsic definition of polynomial convexity on complex projective manifolds. Indeed, the compact set $K=\{[1:e^{i\theta}:e^{-i\theta}]\in\mathbb{CP}^2\,:\,0\leq\theta\leq 2\pi\}$ is polynomially convex when viewed as a subset of $\mathbb{CP}^2\setminus\{z_0=0\}\cong\mathbb{C}^2$, but it is not polynomially convex when viewed as a subset of $\mathbb{CP}^2\setminus\{z_1=0\}\cong\mathbb{C}^2$.

However, viewing $\mathbb{C}^n$ as a subset of complex projective space $\mathbb{CP}^n$, polynomials in $\mathbb{C}^n$ are simply rational functions in $\mathbb{CP}^n$ having poles in the complex hyperplane at infinity. This suggests a generalization of polynomial convexity to a complex projective manifold $X$ by replacing the family of polynomials in~\eqref{PolyHull} with the family of meromorphic functions on $X$ having poles in a fixed hypersurface $Z\subset X$.

Since the hypersurface $Z$ corresponds to a closed positive current of integration $[Z]$ having bidegree $(1,1)$, this motivates a further extension of these ideas, due to Guedj~\cite{Gu99}. Suppose a positive current $T$ of bidegree $(1,1)$ on a complex projective manifold $X$ has \textit{integral cohomology class}; i.e., $[T]\in H^2_{\text{dR}}(X,\mathbb{R})\cong H^2(X,\mathbb{R})$ lies in the image of the morphism $H^2(X,\mathbb{Z})\to H^2(X,\mathbb{R})$ induced by the containment $\mathbb{Z}\hookrightarrow\mathbb{R}$. Then there exists a holomorphic line bundle $L$ on $X$ and a (singular) metric $\varphi$ of $L$ for which $\text{d}\text{d}^c\varphi=T$~\cite[Theorem 5]{Sh95}. We consider compact sets $K\subset X$ for which the hull
\begin{equation}\label{PolyHull2}
	p_T(K)=\left\{z\in X\,:\,|\sigma|e^{-k\varphi}(z)\leq\sup_K|\sigma|e^{-k\varphi}\text{ for all }\sigma\in\Gamma(X,L)\text{ and }k\in\mathbb{N}\right\}
\end{equation}
and $K$ agree.\footnote{Here and throughout we make the slight abuse of notation that $|\sigma|e^{-k\varphi}(z)$ stands for $|\sigma|_{k\varphi}(z)$, the action of the metric $k\varphi$ on the section $\sigma$ at the point $z\in X$.} Note that if $T$ is a current of integration corresponding to a complex hypersurface $Z$, then the objects with respect to which the hull is being taken in~\eqref{PolyHull2} take the form $|\sigma/s^k|$, where $s$ is a holomorphic section of some line bundle on $X$ whose zero divisor coincides with $Z$; that is, they become precisely the family of meromorphic functions on $X$ with poles on $Z$.

There is also the following weaker notion, known as $T$-polynomial convexity.\footnote{It should be noted that, while the notions of convexity corresponding to $p_T(K)$ and $\widehat{K}^T$ are analogous to convexity with respect to polynomial and plurisubharmonic functions, respectively, we call the latter \textit{$T$-polynomial convexity} to be consistent with the literature.} Note that it does not require $T$ to have integral cohomology class.
\begin{definition}
	Let $T$ be a positive closed current of bidegree $(1,1)$ on a complex projective manifold $X$ and let $K$ be a compact subset of $X$. We define the $T$-polynomially convex hull of $K$ by
\[
	\widehat{K}^T=\left\{z\in X\,:\,f(z)\leq\sup_K f\text{ for all }f\in\mathcal{C}_T(X)\text{ such that }\text{d}\text{d}^cf\geq -T\right\},
\]
	where $\mathcal{C}_T(X)$ denotes the set of functions $f\in L^1(X)$ such that $\exp(f+\varphi)$ is continuous whenever $\varphi$ is a local $\text{d}\text{d}^c$-potential of $T$. Note in particular that any $f$ in $\mathcal{C}_T(X)$ is lower semicontinuous.
\end{definition}

Polynomial convexity is of interest, in particular, in view of the Oka--Weil theorem~\cite{St07}. In its simplest formulation, it states that a holomorphic function defined on a neighbourhood of a polynomially convex compact set in $\mathbb{C}^n$ is the uniform limit on $K$ of a sequence of polynomials. This has been extended to $T$-polynomial convexity when $T$ is a current of integration corresponding to a positive divisor~\cite[Theorem 3.4]{Gu99}. Our first result is to extend this further. We need an additional assumption on $T$, called condition (C). This means that $\widehat{K}^T\subset\subset X\setminus\text{Supp }T$ whenever $K\subset\subset X\setminus\text{Supp }T$. Here $c_1:\text{Pic}(X)\to H^2(X,\mathbb{Z})$ denotes the first Chern class homomorphism.

\begin{thm}\label{OkaWeil}
	Let $T$ be a positive closed current of bidegree $(1,1)$ satisfying condition (C) on a projective algebraic manifold $X$. Assume $[T]=c_1(L)$ for some positive holomorphic line bundle $L$ on $X$. Suppose that the compact set $K\subset X\setminus\text{Supp }T$ is $T$-polynomially convex and that $f\in\mathcal{O}(K)$. Then there exist $u_j,v_j\in\Gamma (X,L^{N_j})$, $j\in\mathbb{N}$, so that the meromorphic function $u_j/v_j$ approximates $f$ uniformly on $K$ as $j\to\infty$. Furthermore,
\begin{enumerate}
	\item[(i)] $\frac{1}{N_j}\log|v_j|-\varphi\xrightarrow{j\to\infty} 0$ uniformly on compact subsets of $X\setminus\text{Supp }T$, where $\varphi$ is a (singular) metric of $L$ with $\text{d}\text{d}^c\varphi=T$;
	\item[(ii)] $\frac{1}{N_j}\big[v^{-1}_j(0)\big]\xrightarrow{j\to\infty} T$ in the weak sense of currents;
	\item[(iii)] $v^{-1}_j(0)\xrightarrow{j\to\infty}\text{Supp }T$ in the Hausdorff metric;
	\item[(iv)] $\nu\left(\frac{1}{N_j}\big[ v^{-1}_j(0)\big],z\right)\xrightarrow{j\to\infty}\nu(T,z)$ for all $z\in X$.
\end{enumerate}
\end{thm}
\begin{rem}
	The assumptions are $T$ are not surprising, since the tools available to us require the environment in which we are working to be as ``Stein-like'' as possible. Condition (C) and the positivity of $L$ are analogous to the conditions of holomorphic convexity and separability in the definition of a Stein manifold, respectively.
\end{rem}
In the spirit of the previous remark, our next result shows that, in this context, the notion of $T$-polynomial convexity in $X\setminus\text{Supp }T$ is equivalent to the notion of holomorphic convexity on $X\setminus\text{Supp }T$.

\begin{thm}\label{HoloTPoly}
	Let $T$ be a positive closed current of bidegree $(1,1)$ satisfying condition (C) on a projective algebraic manifold $X$. Assume $[T]=c_1(L)$ for some positive holomorphic line bundle $L$. Then $\widehat{K}^T=\widehat K_{X\setminus\text{Supp }T}$ for all compact subsets $K\subset X\setminus\text{Supp }T$, where $\widehat{K}_{X\setminus\text{Supp }T}$ denotes the holomorphically convex hull of $K$ in the Stein manifold $X\setminus\text{Supp }T$.
\end{thm}
Guedj showed that $X\setminus\text{Supp }T$ is Stein on a compact K\"ahler manifold $X$ whenever $T$ satisfies condition (C) and $T$ is cohomologous to a K\"ahler form~\cite[Theorem 3.8]{Gu99}. If $X$ is a \textit{homogeneous} projective manifold, then every (1,1)-current of integral cohomology class satisfies condition (C), so in such a setting the assumption of condition (C) in the previous statement is redundant. We end the introduction with a natural question.
\begin{question}
	Let $T$ be a positive current of bidegree $(1,1)$ on a complex projective manifold $X$ with $c_1(L)=[T]$ for some positive line bundle $L$. Does $T$ satisfy condition (C)?
\end{question}

\section{$T$-polynomial convexity and holomorphic convexity of $X\setminus\text{Supp }T$}
It is natural to ask under what conditions $p_T(K)=\widehat{K}^T$ holds for all compact subsets $K$. It has been shown that $p_T(K)\subseteq\widehat{K}^T$ with equality whenever $L$ is positive and $X$ is homogeneous~\cite[Proposition 3.2]{Gu99}. Our first observation is that the assumption of homogeneity on $X$ is not necessary for equality. This follows from a somewhat recent extension result, quoted below.

\begin{thm}[{\cite[Theorem B$'$]{CoGuZe2013}}]\label{t.CGZ}
	Let $X$ be a subvariety of a projective manifold $V$ and $L$ be an ample line bundle on $V$. Then any (singular) positive metric of $L|_X$ is the restriction of a (singular) positive metric of $L$ on $V$.
\end{thm}
\begin{cor}\label{LeviProblem}
	Let $X$ be a projective algebraic manifold and $T$ be a positive closed current of bidegree $(1,1)$ on $X$. Assume that $[T]=c_1(L)$ for some positive line bundle $L$, and let $\varphi$ be a (singular) metric on $L$ with $\text{d}\text{d}^c\varphi=T$. Then $p_T(K)=\widehat{K}^T$ for all compact subsets $K\subset X$.
\end{cor}
\begin{proof}
	The inclusion $\widehat{K}^T\subseteq p_T(K)$ is straightforward. Indeed, if $z\not\in p_T(K)$, then there exists a $k$ and a $\sigma\in\Gamma(X,L^k)$ with $|\sigma|e^{-k\varphi}(z)>\sup_K|\sigma|e^{-k\varphi}$. Therefore
\[
	\frac{1}{k}\log |\sigma(z)|-\varphi(z)>\sup_K\left(\frac{1}{k}\log|\sigma|-\varphi\right)
\]
	and so $z\not\in\widehat{K}^T$ since $\tfrac{1}{k}\log|\sigma|(z)-\varphi(z)$ is a member of $\mathcal{C}_T(X)$.

	For the reverse inclusion, fix $z\not\in\widehat{K}^T$. In view of the Kodaira embedding theorem we can assume $X\subset\mathbb{CP}^N$ with $\mathcal{O}(1)|_X=L^k$ for some $k$. Then there exists a positive metric $\psi$ of $L$ on $X$ with $e^{\psi}$ continuous and $(\psi-\varphi)(z)>\sup_K(\psi-\varphi)$~\cite[Proposition 3.2]{Gu99}. Extend $k\psi$ and $k\varphi$ to metrics $\widetilde\psi$ and $\widetilde\varphi$, respectively, on $\mathbb{CP}^N$ by Theorem~\ref{t.CGZ} and set $\widetilde T:=\text{d}\text{d}^c\widetilde\varphi$. Now $z\not\in\widehat{K}^{\widetilde{T}}$, and so we have $z\not\in p_{\widetilde{T}}(K)$, since $\mathbb{CP}^N$ is homogeneous. Thus there exist an integer $M$ and a section $s\in\Gamma(\mathbb{CP}^N,\mathcal{O}(M))$ with
\[
	|s|e^{-M\widetilde\varphi}(z)>\sup_K|s|e^{-M\widetilde\varphi}.
\]
	The restriction of $s$ to $X$ yields a section of $L^{Mk}$ satisfying the above inequality. We conclude that $z\not\in p_T(K)$.
\end{proof}

The proof Theorem~\ref{OkaWeil} requires a result following from the work of Guedj, which will be included for completeness. Owing to its technical nature, the proof will be postponed to the next section.
\begin{thm}[c.f.~{\cite[Theorem 4.1]{Gu99}}]\label{Guedjthm}
	Let $T$ be a positive closed current of bidegree $(1,1)$ satisfying condition (C) on a projective algebraic manifold $X$. Assume $[T]=c_1(L)$ for some positive holomorphic line bundle. Then there exist $N_j\in\mathbb{N}$ and $s_j\in\Gamma(X,L^{N_j})$, $j\in\mathbb{N}$, so that
\begin{enumerate}
	\item[(i)] $\frac{1}{N_j}\log |s_j|-\varphi\xrightarrow{j\to\infty}0$ uniformly on compact subsets of $X\setminus\text{Supp }T$, where $\varphi$ is a metric of $L$ with $\text{d}\text{d}^c\varphi= T$;
	\item[(ii)] $T_j=\frac{1}{N_j}\left[s_j^{-1}(0)\right]\xrightarrow{j\to\infty} T$ in the weak sense of currents;
	\item[(iii)] $s_j^{-1}(0)\xrightarrow{j\to\infty}\text{Supp }T$ in the Hausdorff metric;
	\item[(iv)] $\nu\left(\frac{1}{N_j}\left[s^{-1}_j(0)\right],z\right)\xrightarrow{j\to\infty}\nu(T,z)$ for any $z\in X$.
\end{enumerate}
\end{thm}
\begin{proof}[Proof of Theorem~\ref{OkaWeil}]
	Let $U\subset\subset X\setminus\text{Supp }T$ be a neighbourhood of $K$ on which $f$ is defined. For each $a\in \text{b}U\setminus K$, Corollary~\ref{LeviProblem} provides a positive integer $k$ and a $u\in\Gamma(X,L^k)$ so that
\[
	|u|e^{-k\varphi}(a)>\sup_K|u|e^{-k\varphi}.
\]
	Let $s_j$ be the sections granted by Theorem~\ref{Guedjthm}, and set $T_j=\frac{1}{N_j}[s^{-1}_j(0)]$. By (i), we have
\[
	|u|e^{-\frac{k}{N_j}\log|s_j|}(a)>\sup_K|u|e^{-\frac{k}{N_j}\log|s_j|}
\]
and hence
\[
	\left|\frac{u^{N_j}}{s_j^k}\right|(a)>\sup_K\left|\frac{u^{N_j}}{s^k_j}\right|
\]
	for large $j$. This inequality holds in a neighbourhood of $K$, so there is a neighbourhood of $a$ that is disjoint from $\widehat K^{T_j}$. By compactness of b$U$ we can repeat this process finitely many times to see that $\widehat{K}^{T_j}\subset\subset U$ for large $j$.

	The Kodaira embedding theorem then shows that for every sufficiently large $j$ there exists a sequence of meromorphic functions $\{h_{\ell}/s_j^{P_{j,\ell}}\}_{\ell=1}^{\infty}$ approximating $f$ uniformly on $K$~\cite[Theorem 3.10]{Gu99}. Here $P_{j,\ell}\in\mathbb{N}$ and $h_{\ell}\in\Gamma(X,L^{P_{j,\ell}})$. A diagonalization argument then shows that for every $\varepsilon>0$ there exists a $j_0$ so that
\[
	\sup_K\left(\frac{h_{j}}{s_j^{P_{j,j}}}-f\right)<\varepsilon
\]
whenever $j\geq j_0$.
\end{proof}

The proof of Theorem~\ref{HoloTPoly} follows from Theorem~\ref{OkaWeil}.
\begin{proof}[Proof of Theorem~\ref{HoloTPoly}]
	Let $K$ be a compact subset of $X\setminus\text{Supp }T$.

	First suppose $z\not\in\widehat{K}^T$. Then there exists a $f\in\mathcal{C}_T(X)$ with $\text{d}\text{d}^cf\geq -T$ and
	\begin{equation}\label{pshineq}
	f(z)>\sup_Kf.
\end{equation}
	If $z\in\text{Supp }T$, then clearly $z\not\in\widehat{K}_{X\setminus\text{Supp T}}$, so we can assume that $z\not\in\text{Supp }T$. Note that $f$ is plurisubharmonic on $X\setminus\text{Supp }T$, so the inequality $\eqref{pshineq}$ indicates that $z$ does not belong to the convex hull of $K$ with respect to plurisubharmonic functions on $X\setminus\text{Supp }T$. Since $X\setminus\text{Supp }T$ is Stein~\cite[Theorem 0.4]{Gu99}, this hull coincides with $\widehat{K}_{X\setminus\text{Supp }T}$ (e.g. see Forstneri\v{c}~\cite[Corollary 2.5.3]{Fo17}). Therefore we have $\widehat{K}_{X\setminus\text{Supp }T}\subseteq\widehat{K}^T$.

	For the reverse inclusion, suppose $z\not\in\widehat{K}_{X\setminus\text{Supp }T}$. Condition (C) is satisfied by $T$, so $z\not\in\widehat{K}^T$ whenever $z\in\text{Supp }T$ and we can thus assume $z\not\in\text{Supp }T$. Then there exists a $f\in\mathcal{O}(X\setminus\text{Supp }T)$ with
\begin{equation}\label{holoinq}
|f(z)|>\sup_K|f|.
\end{equation}
	Again, by condition (C), the $T$-polynomially convex hull of $K\cup\{z\}$ is contained in $X\setminus\text{Supp }T$, and an application of Theorem~\ref{OkaWeil} yields sequences $u_j,v_j\in\Gamma(X,L^{N_j})$, $N_j\in\mathbb{N}$, so that the meromorphic functions $\{u_j/v_j\}_{j=1}^{\infty}$ approximate $f$ uniformly on $K\cup\{z\}$ and $v_j$ satisfies conditions (i)-(iv) of the theorem. Since inequality~\eqref{holoinq} is strict, we have $|u_j/v_j|(z)>\sup_K|u_j/v_j|$, or
\[
	|u_j|e^{-\log|v_j|}(z)>\sup_K\left(|u_j|e^{-\log|v_j|}\right)
\]
for large $j$. Furthermore, condition (i) of Theorem~\ref{OkaWeil} asserts that $\frac{1}{N_j}\log|v_j|-\varphi\xrightarrow{j\to\infty}0$ uniformly on $K\cup\{z\}$, which in turn implies
\[
	|u_j|e^{-N_j\varphi}(z)>\sup_K\left(|u_j|e^{-N_j\varphi}\right)
\]
for large $j$.
\end{proof}
\section{Proof of Theorem~\ref{Guedjthm}}

We require the following proposition. Here $E_c(T)=\{z\in X\,:\nu(T,z)\geq c\}$ and $E^+(T)=\{z\in X\,:\,\nu(T,z)\geq 0\}$.
\begin{prop}[c.f. {\cite[Proposition 4.2]{Gu99}}]\label{Guedjprop}
	Let $T$ be a positive closed current of bidegree $(1,1)$ satisfying condition (C) on a projective algebraic manifold $X$ such that $[T]=c_1(L)$ for some positive holomorphic line bundle $L$. Suppose $\varphi$ is a (singular) metric on $L$ with $\text{d}\text{d}^c\varphi=T$, $K$ is a compact subset of $X\setminus\text{Supp T}$ with $\widehat{K}^T=K$, and fix a K\"ahler form $\omega$ on $X$. Then for every open set $V$ with $K\subset V\subset\subset X\setminus\text{Supp }T$ and every $\delta>0$, we can find $M\in\mathbb{N}$ and construct a positive (singular) metric $\psi$ of $L^M$ on $X$ and a section $h\in\Gamma(V,L^M)$ such that
\begin{enumerate}
	\item[(i)] $K\subset\{a\in V\,:\,|h|_{\psi}\geq 1\}=\{a\in V\,:\,|h|_{\psi}\equiv 1\}\subset\subset V$
	\item[(ii)] $\left\|\frac{\psi}{M}-\varphi\right\|_{L^{\infty}(\overline{V})}\leq\delta$ and $\left\|\frac{\psi}{M}-\varphi\right\|_{L^1(X)}\leq\delta$,
	\item[(iii)] $\sup_X\left|\nu\left(\frac{1}{M}\text{d}\text{d}^c\left(\psi\right),\,\cdot\,\right)-\nu(\text{d}\text{d}^c\varphi,\,\cdot\,)\right|\leq\delta$,
	\item[(iv)] $\text{d}\text{d}^c\psi\geq\varepsilon\omega$ in a neighbourhood of $\text{Supp }T$ for some constant $\varepsilon>0$,
	\item[(v)] $\psi$ is continuous in $X\setminus\text{Supp }T$ and smooth on a dense subset of $X\setminus E_{c_0}(T)$ for some $c_0>0$ which can be made arbitrarily small.
\end{enumerate}
\end{prop}
\begin{proof}[Proof of Theorem~\ref{Guedjthm}]
Since $T$ satisfies condition (C), we can find a sequence $\{K_n\}_{n=1}^{\infty}$ of $T$-polynomially convex compact subsets of $X\setminus\text{Supp }T$ which exhaust $X\setminus\text{Supp }T$. Fix a sequence $\{\delta_n\}_{n=1}^{\infty}$ of positive numbers tending to zero, and open neighbourhoods $V_n\subset\subset X\setminus\text{Supp }T$ of $K_n$. Via Proposition~\ref{Guedjprop} we construct $M_n$, positive metrics $\psi_n$ of $L^{M_n}$ and holomorphic sections $h_n$ of $L^{M_n}$ in $V_n$ with the properties (i)--(v).

	Fix a sequence of points $\{a_j\}_{j=1}^{\infty}$, dense in $\text{Supp }T$, such that for all $n\in\mathbb{N}$ we have $a_1,\ldots,a_n\in\text{Supp }T\setminus E_{c_n}(T)$, where $\{c_n\}_{n=1}^{\infty}$ is a sequence of positive numbers converging to zero with $\psi_n$ smooth and $\text{d}\text{d}^c\psi_n>0$ at the points $a_1,\ldots,a_n$. Further, let $\{F_n\}_{n=1}^{\infty}$ be a sequence of compact subsets of $X\setminus E^{+}(T)$ with $\bigcup_{j=1}^{n}F_n=X\setminus E^{+}(T)$ and $K_n\cup\{a_1,\ldots,a_n\}\subset F_n\subset\subset X\setminus E_{c_n}(T)$.

	For the following we will treat $n$ as fixed and consequently omit subscripts of `$n$' to simplify notation. Fix an open covering $\{U_\alpha\}_{\alpha\in A}$ of $X$ by trivializations of $L$ fine enough so that for each $a_j$ there is a $\alpha_j\in A$ so that $a_j\in U_{\alpha_j}$ and $U_{\alpha_j}$ contains no other elements of the discrete set $\{a_1,\ldots,a_n\}$.

	Since $\text{d}\text{d}^c\psi>0$ on $\text{Supp }T$, there are holomorphic polynomials $P_j$ so that $\text{Re }(P_j)(a_j)=\psi_{\alpha_j}(a_j)$ and $\psi_{\alpha_j}(z)-\text{Re }(P_j)(z)\geq c_j|z-a_j|^2$ for some $c_j>0$ in a local coordinate patch $W_j$ of $a_j$ with $W_j\subset U_{\alpha_j}$; we can further choose the $W_j$ small enough so that $W_j\cap U_{\beta}=\varnothing$ for all $\beta\in A\setminus \{\alpha_j\}$.

	Let $\chi_j\in\mathcal{C}_0^{\infty}(W_j)$, $0\leq\chi_j\leq 1$, be a bump function that is identically equal to one in a neighbourhood of $a_j$ and define smooth sections $f_j=\{f_j^{\alpha}\}_{\alpha\in A}$ of $L^{NM}$ by
\[
f_j^{\alpha}=\begin{cases}
	\chi_j e^{NP_j}, &\text{whenever }\alpha=\alpha_j\\
	0, &\text{whenever }\alpha\neq\alpha_j
\end{cases},
\]
	where $N$ is a large constant that will be chosen momentarily. Let $\xi:X\to [0,1]$ be another smooth bump function which is identically equal to one on a neighbourhood of $K'=\left\{z\in X\,:\,|h|e^{-\psi}(z)\geq 1\right\}$ and has support disjoint from $\text{Supp }T$ and the supports of $\chi_1,\ldots,\chi_n$.

	Define $u=\xi h^N+\sum_{j=1}^{n}f_j$. This is a smooth global section of $L^{NM}$ on $X$ with the properties:
\begin{align}
	|u|_{N\psi}=1\quad\quad&\text{on}\quad K\cup\{a_1,\ldots,a_n\}\label{prop1}, \text{ and}\\
	|u|_{N\psi}\leq 1\quad\quad &\text{on $X$ and strict outside a neighbourhood of } K\cup\{a_1,\ldots, a_n\}\label{prop2}
\end{align}
The smooth $\bar\partial$-closed form $\bar\partial u$ is a $(0,1)$-form with values in $L^{NM}$, or alternatively, a smooth $\bar\partial$-closed $(\dim X,1)$-form with values in $L^{NM}\otimes K^*_X$, where $K_X$ denotes the canonical bundle on $X$.

	Set $N=N_1+N_2$, where $N_2$ is chosen so that $L^{N_2M}\otimes K^*_X$ is positive, and $N_1$ is a large positive integer to be determined soon. Fix a K\"ahler form $\omega$ on $X$, $\varepsilon>0$, and a smooth metric $G$ of $L^{N_2M}\otimes K^*_X$ with $\text{d}\text{d}^cG\geq\varepsilon\omega$. We solve a $\bar\partial$-problem with $L^2$-estimates associated to the metric $\theta=N_1\psi+G$~\cite[Theorem 3.1]{De92b} to find a smooth $(\dim X,0)$-form $v$ with values in $L^{NM}\otimes K^*_X$ satisfying $\bar\partial v=\bar\partial u$ and
\begin{equation*}
	\int_X|v|^2e^{-2\theta}\text{d}V_{\omega}\leq\frac{1}{\varepsilon}\int_X|\bar\partial u|^2e^{-2\theta}\text{d}V_{\omega},
\end{equation*}
	where $\text{d}V_{\omega}$ denotes the K\"ahler volume element $\frac{1}{\dim(X)!}\omega^{\text{dim}(X)}$. Since $\text{Supp }(\bar\partial\xi)\subset\{z\in X\,:\, |h|_{\psi}(z)<1\}$ and $\text{Supp }(\bar\partial\chi_j)\subset\left\{z\in W_j\,:\,|e^{MP_j}|_{-\psi}<1\right\}$, we can find a number $a<1$ so that $|\bar\partial u|^2e^{-2N\psi}\lesssim a^{2N_1}$, where the implied constant is independent of $N_1$. Then
\begin{equation}\label{uniformbound}
	\int_X|v|^2e^{-2\theta}\text{d}V_\omega\lesssim a^{2N_1}.
\end{equation}
	Now since $\psi$ is continuous on $X\setminus E_{c}(T)$, it is uniformly continuous on a continuous on a compact neighbourhood of $F$ which is compact in $X\setminus E_c(T)$. For a fixed $z\in F$, choose an $r>0$ small enough so that $e^{\eta}a<1$, where $\eta=\sup_{w\in B(z,r)}|\psi(w)-\psi(z)|$ represents the uniform oscillation of $\psi$ on $B(z,r)$. A lemma of H\"ormander--Wermer~\cite[Lemma 4.4]{HoWe1968} yields
\begin{align}
	|v(z)|^2&\lesssim r^2\sup_{w\in B(z,r)}|\bar\partial v(w)|^2+r^{-2\dim(X)}\|v\|^2_{L^2(B(z,r))}\nonumber\\
	&=r^2\sup_{w\in B(z,r)}\left(e^{2N\psi(w)}|\bar\partial v(w)|^2e^{-2N\psi(w)}\right)+r^{-2\dim(X)}\|v\|^2_{L^2(B(z,r))}\nonumber\\
	&\leq r^2 e^{2N\psi(z)}e^{2N\eta}\sup_{w\in B(z,r)}\left(|\bar \partial v(w)|^2e^{-2N\psi(w)}\right)+r^{-2\dim(X)}\|v\|^2_{L^2(B(z,r))}\nonumber\\
	&\lesssim e^{2N\psi(z)}e^{2N_1\eta}\left(\sup_{B(z,r)}\left(|\bar\partial v|^2e^{-2N\psi}\right)+\int_{B(z,r)}|v|^2e^{-2\theta}\text{d}V_{\omega}\right)\nonumber\\
	&\lesssim (e^{\eta}a)^{2N_1}e^{2N\psi(z)}\label{largeestimate},
\end{align}
	where again the implied constants are independent of $N_1$. Here $B(z,r)$ stands for the pullback of a Euclidean ball via a coordinate chart.

	Choose $r$ so that $e^{\eta}a<1$ and $N_1$ so that $|v|e^{-N\psi}\leq\tfrac{1}{n-1}$ on $F_n$. Set $S=\tfrac{n-1}{n}(u-v)$. Then
\begin{align}
	|S|e^{-N\psi}&\leq\frac{n-1}{n} \left|u\right|e^{-N\psi}+\frac{n-1}{n} \left|v\right|e^{-N\psi}\leq 1\quad\text{on $F_n$}\label{prop3}
\shortintertext{and likewise}
	|S|e^{-N\psi}&\geq\frac{n-1}{n}\left|u\right|e^{-N\psi}-\frac{n-1}{n} \left|v\right|e^{-N\psi}\geq\frac{n-2}{n}\quad\text{on $K_n\cup\{a_1,\ldots,a_n\}$}\label{prop4}
\end{align}
	by \eqref{prop1} and \eqref{prop2} above.

	Given $\varepsilon>0$ and a compact subset $A\subset X\setminus\text{Supp }T$, by \eqref{prop3} and \eqref{prop4} we have $\smash{\|\tfrac{1}{N_n}\log|S_n|-\psi_n\|_{L^{\infty}(A)}<\varepsilon/2}$. By choice of $n$ large enough so that both $\delta_n$ is smaller than $\varepsilon/2$ and $A\subset K_n\cap F_n$, one also has $\smash{\|\tfrac{1}{M_n}\psi-\varphi\|_{L^{\infty}(A)}<\varepsilon/2}$ by property (ii) of Proposition~\ref{Guedjprop}. The triangle inequality then shows conclusion (i).

	For conclusions (ii)--(v), we will show
	\begin{align}
		\nu\left(\frac{1}{N_n}\text{d}\text{d}^c\left(\log|S_n|\right),z\right)&\geq\left(1-\frac{1}{\sqrt{N_n}}\right)\nu\left(\text{d}\text{d}^c\psi_n,z\right)-\frac{1}{N_n}\quad\text{for all }z\in E_{c_n}\left(\text{d}\text{d}^c\psi_n\right)\label{lelongineq}
		\shortintertext{and}
		\int_X|S_n|e^{-N_n\psi_n}&\leq C,\quad\quad\text{for $C>0$ independent of $n$}.\label{uniformbound2}
	\end{align}
	Note that $u$ is identically zero in a neighbourhood of any point $z\in E_{c_n}(\text{d}\text{d}^c\psi_n)$, so $v$ is holomorphic there and the finiteness of the integral $\int_X|v|^2e^{-2\theta}\text{d}V_{\omega}$ forces $v$ to vanish to an order greater than or equal to $N_1\nu(\text{d}\text{d}^c\psi_n,x)-1$ at $z$ (c.f. Kiselman \cite[Theorem 3.2]{Ki94}). We thus have
\[
	\nu\left(\frac{1}{N}\text{d}\text{d}^c\left(\log|S|\right),z\right)\geq\frac{N_1}{N}\nu\left(\text{d}\text{d}^c\psi_n,z\right)-\frac{1}{N}\geq\left(1-\frac{1}{\sqrt{N}}\right)\nu\left(\text{d}\text{d}^c\psi_n,z\right)-\frac{1}{N},
\]
for $N_1$ large enough, since $N=N_1+N_2$ and $N_2$ is fixed. This establishes the inequality~\eqref{lelongineq}. To see the inequality~\eqref{uniformbound2}, note that we have $\int_X|v|^2e^{-2\theta}\text{d}V_\omega\lesssim 1$ from~\eqref{uniformbound}, where the implied constant is independent of $n$; from this it follows that $\int_X|v|^2e^{-2N\psi}\text{d}V_{\omega}\lesssim 1$ as well. From this, we obtain
\[
	\int_X|S|e^{-N\psi}\text{d}V_{\omega}\lesssim\int_X|S|^2e^{-2N\psi}\text{d}V_\omega\lesssim 1.
\]

	To see (ii), the inequalities \eqref{prop3} and \eqref{prop4}, along with the Lelong--Poincar\'{e} equation and (ii) of Proposition~\ref{Guedjprop} implies that $T_n=\text{d}\text{d}^c\left(\frac{1}{N_nM_n}\log |S_n|\right)$ converges weakly towards $T$ in $X\setminus E^+(T)$. Since the decomposition theorem of Siu~\cite{Siu74} gives $T=\sum_i\lambda_i[Z_i]+R$, where $\lambda_i$ are positive constants, $Z_i$ are hypersurfaces, and $R$ is a residual current with $E_c(R)\geq 2$ for all $c>0$, the Hausdorff dimension of $E^+(R)=E^+(T)\setminus\bigcup_{i}Z_i$ is less than or equal than $2\dim(X)-4$, hence $T_n$ actually converges towards $T$ on $X\setminus\bigcup_iZ_i$ (e.g. \cite{FoSi95}).

	Also, note that $\limsup_{n\to\infty}\nu(T_n,x)\leq\nu(T,x)$ for all $x\not\in\bigcup_{i}Z_i$ since $T_n$ converges to $T$ in the weak sense of currents there (c.f. Demailly \cite[Proposition III.5.12]{De12}); hence the inequality~\eqref{lelongineq} and condition (iii) of Proposition~\ref{Guedjprop} gives that $\nu(T_n,x)\to\nu(T,x)$ for all $x\in E^{+}(T)\setminus\bigcup_iZ_i$ and hence $\nu(T_n,x)\to\nu(T,x)$ for all $x\not\in\bigcup_iZ_i$. Now, the sequence $\{T_n\}_n$ is bounded in the sense of currents by the inequality~\eqref{uniformbound2} and (ii) of Proposition~\ref{Guedjprop} (see Demailly~\cite[pg. 19]{De12}), so it admits a subsequence weakly converging to a positive closed current $T'$ of bidegree $(1,1)$ on $X$. Since $T'=T$ on $X\setminus\bigcup_iZ_i$ and $T'\geq T$ on $\bigcup_iZ_i$ by the inequality~\eqref{lelongineq}, it follows that $T=T'$ on all of $X$ because $c_1(L)=[T]=[T']$ and the fact that $X$ it is compact and K\"ahler. Therefore $T_n$ converges weakly towards $T$ on $X$ and $\nu(T_n,x)\to\nu(T,x)$ for all $x\in X$.

	Lastly, since $|S_n|>0$ on $K_n$ and $T_n\to T$, the varieties $\{S_n=0\}$ converge towards $\text{Supp }T$ in the Hausdorff metric.
	\end{proof}
\begin{rem}
	It should be noted that, while the focus of this note is on complex projective manifolds, $T$-polynomial convexity can be defined on Stein manifolds. In fact, on a Stein manifold, every line bundle is positive and every positive closed (1,1)-current satisfies condition (C), so Theorems~\ref{OkaWeil} and~\ref{HoloTPoly} can be stated in that setting with some hypotheses removed. The methods of proof are identical, with one exception: Stein manifolds are not compact, so in place of compactness of the complex projective manifold one instead takes advantage of the compactness of a large sublevel set of some strictly plurisubharmonic exhaustion function along with an appropriate patching argument.
\end{rem}


\printbibliography

\end{document}